\documentclass[reqno]{amsart}
\usepackage[T1]{fontenc}
\usepackage[latin1]{inputenc}
\usepackage[english]{babel}
\usepackage[babel]{csquotes}

\usepackage{cite}
\usepackage{amssymb}
\usepackage{amsmath}
\usepackage{amsthm}
\usepackage{latexsym}
\usepackage{graphicx}
\usepackage{mathrsfs}
\usepackage{bbm}
\usepackage{verbatim}
\usepackage[colorlinks=true, pdfstartview=FitV, linkcolor=blue, citecolor=blue, urlcolor=blue]{hyperref}

\usepackage{enumerate, enumitem}
\usepackage[usenames,dvipsnames]{color}

\newtheorem{thm}{Theorem}[section]

\newtheorem{lem}[thm]{Lemma}
\newtheorem{prop}[thm]{Proposition}

\newtheorem{remark}[thm]{Remark}

\numberwithin{equation}{section}

\def\enne{\mathbb{N}}
\def\zeta{\mathbb{Z}}

\def\erre{\mathbb{R}}
\def\R{\mathbb{R}}

\def\P{\mathbb{P}}

\def\H{\mathcal H}
\def\E{\mathop{{}\mathbb{E}}}

\def\cL{\mathscr{L}}
\def\cF{\mathscr{F}}

\def\cP{\mathscr{P}}

\def\OO{\mathcal{O}}
\renewcommand{\d}{{\mathrm d}}

\def\beq{\begin{equation}}
\def\eeq{\end{equation}}

\def\to{\rightarrow}

\def\embed{\hookrightarrow}

\def\norm #1{\left\|#1\right\|}

\def\sp #1#2{\left<#1,#2\right>}
\newcommand\ip\sp

\begin{document}
\title[Separation for the stochastic Allen-Cahn equation]
{A note on regularity and separation\\ for the
stochastic Allen-Cahn equation\\
with logarithmic potential}

\author{Carlo Orrieri}
\address[Carlo Orrieri]{Department of Mathematics, 
Universit\`a di Pavia, Via Ferrata 1, 27100 Pavia, Italy}
\email{carlo.orrieri@unipv.it}
\urladdr{http://www-dimat.unipv.it/orrieri}

\author{Luca Scarpa}
\address[Luca Scarpa]{Department of Mathematics, Politecnico di Milano, 
Via E.~Bonardi 9, 20133 Milano, Italy.}
\email{luca.scarpa@polimi.it}
\urladdr{https://sites.google.com/view/lucascarpa}



\maketitle

%\centerline{{\sl dedicated to Pierluigi Colli on the occasion of his 65th birthday}}

\begin{abstract}
We prove refined space-time regularity for the classical 
stochastic Allen-Cahn equation with logarithmic potential.
This allows to establish a random separation property, 
i.e.~that the trajectories of the solution are strictly 
separated from the potential barriers.
The present contribution extends
the results obtained in \cite{BOS},
where separation was proved only for 
$p$-Laplace operators with $p$ greater than the space-dimension.
\end{abstract}

\setcounter{tocdepth}{1}
\tableofcontents

%%%%%%%%%%%%%%%%%%%%%%%%%%%%%%%%%%%%%%%%%%%%%%%%

\section{Introduction}
\setcounter{equation}{0}
\label{sec:intro}

The focus of this paper is a qualitative study of the 
stochastic Allen-Cahn equation with logarithmic potential 
in the direction of regularity of solutions and 
consequences on the separation property.

The Allen-Cahn equation is a well-established model
in the context of phase-separation phenomena in several fields 
such as physics and biology. 
From the modelling point of view, it describes the evolution of 
binary mixtures, or more generally of materials exhibiting 
two different phases occupying a given
domain $\OO \subset \R^d$, with $d=1,2,3$. 
The difference of volume fractions 
of the two phases is represented by 
a phase variable $u$, taking values in $[-1,1]$. The 
pure phases are identified with the
regions $\{u=-1\}$ and $\{u=1\}$, 
and it is supposed that they are separated by a
narrow diffuse interface $\{-1<u<1\}$.

The Allen-Cahn equation reads as
\beq
  \label{eq:ac}
  \partial_t u - \Delta u + F'(u) = f \quad\text{in } (0,T)\times\OO\,,
\eeq
where $T>0$ is a fixed final time and $f$ is a given distributed source.
The equation is classically coupled with 
homogenous Dirichlet or Neumann boundary conditions and a given initial datum.
The thermodynamically relevant choice of the potential $F$ is
the logarithmic one, also know as Flory-Huggins potential, defined as
\beq
  \label{F_log}
  F_{\log}(r):=\frac{\theta}{2}((1+r)\log(1+r)+(1-r)\log(1-r))-\frac{\theta_0}{2}r^2 \,,
  \quad r\in[-1,1]\,,
\eeq
where $0 < \theta < \theta_0$ are given constants. 
It is well-known that the Allen-Cahn equation is the 
$L^2(\OO)$-gradient flow of the free energy functional 
\beq
  \label{free_en}
  \mathcal E(u):=\frac12\int_\OO|\nabla u|^2 + \int_\OO F(u)\,.
\eeq
The dynamics of the Allen-Cahn equation result then 
from the competition between the minimisation of the two
terms of the free energy functional, namely 
the concentration towards the two minima of $F$ and 
the avoidance of high oscillations.

One of the main features of the logarithmic potential \eqref{F_log}
is that its two global minima are inside the physical range $(-1,1)$
and its derivative is singular at $\pm1$. Consequently, 
one of the crucial issues in the study of the model
concern the separation of the solutions from the potential barriers $\pm1$.
More precisely, if the initial concentration is an actual mixture 
of the two phases, namely if
\[
\exists\,\delta_0\in(0,1):\quad 
|u_0(x)|\leq 1-\delta_0 \quad\text{for a.e.~}x\in\OO\,,
\]
it is possible to prove that the entire trajectory $u$
is also separated from $\pm1$ in the sense that 
\[
  \exists\,\delta\in(0,1):\quad
  |u(t,x)|\leq 1-\delta \quad\text{for a.e.~}x\in\OO\,, \quad\forall\,t\in[0,T]\,.
\]
The investigation of the separation property 
is relevant both from the modelling perspective, 
since the convex contribution of the potential
can be used to measure the macroscopic mixing of the constituents,
and from the mathematical viewpoint.
Indeed, if the trajectory $u$ is separated, the action of $F'$
on $u$ behaves in a Lipschitz way, despite 
of the singularity at $\pm1$. {\em A posteriori}, this justifies
the typical polynomial approximation of the potential $F$
which is often employed in the literature.
For further detail on the separation property 
we refer to \cite[Sec.~6.2]{GGG_sep}.

The literature on deterministic diffuse-interface models 
is extremely developed. With no sake of completeness,
we refer to the contributions 
\cite{cal-colli, colli-sprek-optACDBC}, and the references therein, 
dealing with well-posedness and optimal control of the 
Allen-Cahn equation. These results are strongly based on the separation
property, which follows from the maximum principle in the deterministic case.
Concerning fourth-order models such as the Cahn-Hilliard equation,
the separation property is a much more delicate issue due 
to the lack of the maximum principle.
In the two-dimensional case this was investigated in the papers
\cite{GGG, GGG_sep}, while the three-dimensional case has
been recently established in the contributions 
\cite{G_sep, P_sep}. Further results for 
more general models can be found in
\cite{bcst2, GGM, GGW}.

Although the deterministic approach is 
satisfactory in capturing several effects occurring in phase-separation,
it fails in describing all the possible microscopic perturbations 
due to, e.g., configurational, vibrational, and magnetic oscillations. 
These can be included in the model by switching to 
a stochastic description involving a random source term.
Usually, for mathematical purposes this is achieved by 
introducing a cylindrical Wiener process in the equation 
for the phase-variable. 
The resulting stochastic Allen-Cahn equation reads as
\beq
  \label{eq:ac_s}
  \d u -\Delta u\,\d t + F'(u)\,\d t = \mathcal H(u)\,\d W \,, 
\eeq
where $W$ is a given cylindrical Wiener process and 
$\H$ is a suitable stochastically-integrable operator.
The well-posedness of the stochastic equation with singular potential,
as well as the separation property, is significantly 
more involved due to absence of a maximum principle.
The first well-posedness result
with logarithmic potential can be found in
\cite{bertacco}, where a suitable degeneracy 
of the multiplicative noise is exploited to 
compensate the blow-up of the potential.
Analogous techniques have been employed in
the context of stochastic Cahn-Hilliard equations 
and thin-films equations in 
\cite{dar-gess-gnann-grun, grun-metzger, scar-mobility}.

For what concerns random separation property of the 
stochastic Allen-Cahn model, 
the first contribution
is given in the work \cite{BOS}, where 
a class of $p$-Laplace Allen-Cahn equations is considered with $p>d$.
The presence of the $p$-Laplace operator allows to 
obtain suitable spatial H\"older regularity of the solution, 
which in turn ensures pointwise evaluation in space and time.
By combining this information with suitable estimates 
on the derivatives of the potential, 
an adaptation of an argument of \cite{sch-seg-stef}
produces the separation property 
by a contraction argument.
More precisely, one is able to show that 
if  
\[
\exists\,\delta_0\in(0,1):\quad 
|u_0(x)|\leq 1-\delta_0 \quad\text{for a.e.~}x\in\OO\,,
\]
then for almost every $\omega \in \Omega$
there exists $\delta(\omega)\in(0,1)$ such that 
\[
\sup_{(t,x)\in[0,T]\times\overline\OO}|u(\omega,t,x)|
 \leq 1-\delta(\omega)\,.
\] 
This shows that almost every trajectory 
is strictly separated from the barriers $\pm1$,
but the threshold of separation is $\omega$-dependent
and identifies naturally a random variable.
Further qualitative properties of the separation layer
are obtained in \cite{BOS}, such as
probability estimates, exponential decay, and convergence
towards the deterministic model as the noise vanishes.

The technical requirement $p>d$ in \cite{BOS}
unfortunately excludes the classical Allen-Cahn equation (with $p=2$)
in dimension 2 and 3. The main goal of the present 
contribution is to fill the gap and 
establish the random separation property 
for the case $p=2$ up to dimension 3.
The idea is to first recover some refined spatial H\"older regularity 
on the solution via suitable $H^2$-estimates:
these are obtained by a careful application of It\^o formula
for the derivatives of
the potential.

To summarise, in this work we focus on the problem
\begin{equation}\label{eq:intro}
\begin{aligned}
  \d u - \Delta u\,\d t + F'(u)\,\d t = \mathcal H(u)\,\d W \qquad&\text{in } (0,T)\times\OO\,,\\
  u=0 \qquad&\text{in } (0,T)\times\partial\OO\,,\\
  u(0)=u_0 \qquad&\text{in } \OO\,.
\end{aligned}
\end{equation}
where $F'$ is singular at the endpoints $\pm1$, so that the 
logarithmic potential case is included.

The paper is structured in the following way. 
Section~\ref{sec:main} presents the general setting and 
the statements of the main results, as well as a comment on
possible further developments,
Section~\ref{sec:reg} contains the proof of the refined regularity estimates, 
and Section~\ref{sec:proof2} deals with the proof of the separation result.

%%%%%%%%%%%%%%%%%%%%%%%%%%%%%%%%%%%%%%%%%%%%%%%%

\section{Setting and main results}
\label{sec:main}

\subsection{Setting and notation}
Let $T>0$ be a fixed final time, let $(\Omega,\cF,\P)$ be
a probability space endowed with a filtration $(\cF_t)_{t\in[0,T]}$
which satisfies the usual conditions, 
and let $W$ be a 
cylindrical Wiener process on a separable Hilbert space $U$.
The progressive sigma algebra on $\Omega\times[0,T]$
is denoted by $\cP$.
Throughout the paper, $(e_k)_{k\in\enne}$
denotes a complete orthonormal system of $U$.
Furthermore, let $\OO\subset\erre^d$, $d\in\{1,2,3\}$, be a bounded domain 
with Lipschitz boundary $\partial\OO$, and 
define the functional spaces 
\[
  H:=L^2(\OO)\,, \qquad V:=H^1_0(\OO)\,, \qquad Z:= H^2(\OO)\cap H^1_0(\OO)\,,
\]
endowed with their classical norms $\norm{\cdot}_H$,
$\norm{\cdot}_{V}$, and $\norm{\cdot}_{Z}$, respectively. 
We identify $H$ with its dual $H^*$, so that
we have the dense, continuous, and compact inclusions
\[
  Z\embed V \embed H \embed V*\embed Z^*\,.
\] 
We recall that the Laplacian operator with homogeneous Dirichlet 
boundary conditions can be seen as an unbounded linear operator 
$-\Delta$ on $H$, with effective domain $Z$.
Eventually, for every $R>0$ we introduce the sublevel set
\[
  B_R:=\{\varphi\in H:\; |\varphi|\leq R \quad\text{a.e.~in } \OO\}\,.
\]

For every Hilbert spaces $E_1$ and $E_2$, 
we use the symbol 
$\cL^2(E_1,E_2)$ to denote the
space of Hilbert-Schmidt operators from $E_1$ to $E_2$.
For every $s,r\in[1,+\infty]$ and for every Banach space $E$
we use the classical symbols $L^s(\Omega; E)$ and $L^r(0,T; E)$
for the spaces of Bochner-integrable functions 
on $\Omega$ and $(0,T)$, respectively.
In the particular case where
$s,r\in[1,+\infty)$, we employ the notation 
$L^s_\cP(\Omega;L^r(0,T; E))$ to stress 
progressive measurability of the processes.
We recall also that
when $s\in(1,+\infty)$,
$r=+\infty$, and $E$ is separable and reflexive,
the space 
\[
  L^s_w(\Omega; L^\infty(0,T; E^*)):=
  \left\{v:\Omega\to L^\infty(0,T; E^*) \text{ weakly* meas.}\,:\,
  \E\norm{v}_{L^\infty(0,T; E^*)}^s<\infty
  \right\}\,,
\]
satisfies by
\cite[Thm.~8.20.3]{edwards} the identification
\[
L^s_w(\Omega; L^\infty(0,T; E^*))=
\left(L^{s/(s-1)}(\Omega; L^1(0,T; E))\right)^*\,.
\]
Finally, the symbol $C^0_w([0,T]; E)$ denotes the space of continuous functions from $[0, T]$ to the Banach space $E$ endowed with the weak topology.

\subsection{Assumptions and main results}
We recall that the logarithmic potential 
is defined as in \eqref{F_log} as
\[
  F_{\log}(r)=\frac{\theta}{2}((1+r)\log(1+r)+(1-r)\log(1-r))-\frac{\theta_0}{2}r^2 \,,
  \quad r\in[-1,1]\,,
\]
where $0 < \theta < \theta_0$ are given constants, and
that the derivatives 
of $F_{\log}$ are given by 
\[
  F'_{\log}(r)=\frac\theta2\log\frac{1+r}{1-r} - \theta_0r\,, \qquad
  F''_{\log}(r)=\frac\theta{1-r^2} - \theta_0\,, \qquad r\in(-1,1)\,.
\]

We need to introduce the following functions, for $s\geq1$:
\beq
  \label{Gs}
  G_s:(-1,1)\to\erre\,, \qquad G_s(x):=\frac1{(1-x^2)^s}\,, \quad x\in(-1,1)\,.
\eeq
These will be crucial also to quantitatively measure how $u$ accumulates 
towards $\pm1$.

The main assumptions that we require cover very large classes of potentials,
and in particular they apply to the logarithmic one.
We assume the following setting.

\begin{enumerate}[start=1,label={{(H\arabic*})}]

\item \label{H1}
The function $F:(-1,1)\to \erre$ satisfies
\begin{itemize}
  \item $F$ is of class $C^2$, nonnegative, and $F'(0)=0$;
  \item $\lim_{r\to(\mp1)^{\pm}}F'(r)={\mp}\infty$;
  \item there are $C_F>0$ and $s_F\geq1$
  such that 
  \[
  -C_F\leq F''(r)\leq C_F(1+G_{s_F}(r)) \quad\forall\,r\in(-1,1)\,.
  \]
\end{itemize}

\item \label{H2}
The sequence of functions $(h_k)_{k\in\enne}$ satisfies 
\begin{itemize}
  \item $h_k\in W_0^{1,\infty}(-1,1)$ for every $k\in\enne$;
  \item $F''h_k^2\in L^\infty(-1,1)$ for every $k\in\enne$;
  \item $C_{1,\mathcal H}^2:=\sum_{k=0}^\infty
  (\|h_k\|_{W^{1,\infty}(-1,1)}^2
  +\|F''h_k^2\|_{L^\infty(-1,1)})
  <+\infty$.
\end{itemize}
Let us notice that hypothesis \ref{H2} ensures that the 
operator
\[
  \H: B_1\to \cL^2(U,H)\,,\qquad
  \H(v)[e_k]:=h_k(v)\,, \quad v\in B_1\,,\quad k\in\enne\,,
\]
is well-defined and $C_{1,\H}$-Lipschitz-continuous 
with respect to the topology of $H$.
\end{enumerate}

\begin{remark}
  Note that the logarithmic potential satisfies \ref{H1}--\ref{H2} with 
  natural choices of the constant $C_F$ and with $s_F=1$.
\end{remark}

We first recall the well-posedness result for equation \eqref{eq:intro}:
this is proved in \cite[Thm.~2.1]{bertacco} in the case $p=2$
(see also the techniques 
in \cite{BOS} for non-random initial datum).
\begin{prop}\label{prop:WP}
  Assume \ref{H1}--\ref{H2}, let $p\geq2$, and let
  \[
  u_0\in L^p(\Omega,\cF_0; V)\,, \qquad
  F(u_0)\in L^{\frac p2}(\Omega,\cF_0; L^1(\OO))\,.
  \]
  Then, there exists a unique 
  \[
  u \in L^p_\cP(\Omega; C^0([0,T]; H))\cap 
  L^p_w(\Omega; L^\infty(0,T; V))\cap
  L^p_\cP(\Omega; L^2(0,T; Z))
  \]
  such that 
  \[
  F'(u) \in L^p_\cP(\Omega; L^2(0,T; H))\,,
  \]
  and 
  \begin{align*}
  u(t) -\int_0^t\Delta u(s)\,\d s + \int_0^tF'(u(s))\,\d s
  =u_0 + \int_0^t\mathcal H(u(s))\,\d W(s)
  \quad\forall\,t\in[0,T]\,,\quad\P\text{-a.s.}
  \end{align*}
  Moreover, the solution map $u_0\mapsto u$ is Lipschitz-continuous 
  from the space $L^p(\Omega,\cF_0; H)$ to the space
  $L^p_\cP(\Omega; C^0([0,T]; H)\cap L^2(0,T; V))$.
\end{prop}

Our first result focuses on space-time regularity for the solution, for which
we need to strengthen assumption \ref{H2}
by adding the following requirements.
\begin{enumerate}[start=3,label={{(H\arabic*})}]

\item \label{H3}
The sequence of functions $(h_k)_{k\in\enne}$ satisfies 
\begin{itemize}
  \item $s_0\in\enne$ is such that $s_0\geq d s_F-1$;
  \item $h_k\in W_0^{1+2s_0,\infty}(-1,1)$ for every $k\in\enne$;
  \item $C_{2,\mathcal H}^2:=\sum_{k=0}^\infty
  \|h_k\|_{W^{1+2s_0,\infty}(-1,1)}^2
  <+\infty$.
\end{itemize}
\end{enumerate}

\begin{thm}[Regularity]
  \label{thm:reg}
  Assume \ref{H1}--\ref{H3}, let $p>2$, and let
  \[
  u_0\in L^p(\Omega,\cF_0; Z)\,, \qquad
  \exp\left(\|G_{s_0}(u_0)\|_{L^1(\OO)}\right) \in 
  L^q(\Omega,\cF_0)\quad\forall\,q\geq1\,.
  \]
  Then, the unique solution $u$ given by Proposition~\ref{prop:WP}
  satisfies 
  \begin{align*}
  u\in L^2_\cP(\Omega; C^0([0,T]; V))\cap L^2_w(\Omega; L^\infty(0,T; Z))
  \cap L^2_\cP(\Omega; L^2(0,T; H^3(\OO)))\,.
  \end{align*}
\end{thm}

By exploiting the regularity obtained in Theorem~\ref{thm:reg}
we prove the separation property. This will only need a 
further requirement on the magnitude of $s_0$, depending on
the space dimension.

\begin{thm}[Separation property]
  \label{thm2}
  Assume \ref{H1}--\ref{H3}, let $p>2$, and let
  \[
  u_0\in L^p(\Omega,\cF_0; Z) 
  \]
  be strictly separated from $\pm1$, i.e.
  \[
  \exists\,\delta_0\in(0,1):\quad
  |u_0|\leq1-\delta_0 \quad\text{a.e.~in } \Omega\times\OO\,.
  \]
  Suppose also that 
  \[
  s_0>\begin{cases}
  2 \quad&\text{if } d=2\,,\\
  6\quad&\text{if } d=3\,.
  \end{cases}
  \]
  Then, if $u$ is the unique solution given by Proposition~\ref{prop:WP} 
  and Theorem~\ref{thm:reg}, it holds that 
  \[
  \P\left\{\exists\,\delta\in(0,\delta_0]:\;\sup_{(t,x)\in[0,T]\times\overline\OO}
  |u(t,x)|\leq1-\delta\right\}=1\,,
  \]
  namely, $u$ is strictly separated from $\pm1$ almost surely.
\end{thm}

\begin{remark}
  Note that the separation of $u_0$ assumed in Theorem~\ref{thm2}
  guarantees straightaway the existence of the exponential moments 
  required in Theorem~\ref{thm:reg}. Hence, the corresponding 
  unique solution $u$ enjoys the additional regularity of Theorem~\ref{thm:reg}.
\end{remark}

  Let us point out that by 
  Theorem~\ref{thm:reg} one has that 
  $u(\omega)\in C^0([0,T]; V)\cap L^\infty(0,T; Z)$ 
  for $\P$-almost every $\omega\in\Omega$. This implies that 
  $u(\omega)\in C^0_w([0,T]; Z)$ 
  for $\P$-almost every $\omega\in\Omega$.
  Since one has $Z\embed C^{0,\alpha}(\overline\OO)$ 
  for every $\alpha\in(0,\frac12)$ in dimension $d=1,2,3$, it is possible to evaluate
  almost every trajectory of $u$ pointwise in space and time, and
  \beq
  \label{pointwise}
  \norm{u(\omega)}_{L^\infty((0,T)\times\OO)}=
  \sup_{(t,x)\in[0,T]\times\overline\OO}|u(\omega,t,x)|
  \quad\text{for $\P$-a.e.~}\omega\in\Omega\,.
  \eeq
  These considerations, together with 
  Theorem~\ref{thm2}, ensure that 
\[
\exists\,\Omega_0\in\cF\,,\quad\P(\Omega_0)=1:\quad
\forall\,\omega\in\Omega_0\,,
\quad\exists\,\delta(\omega)\in(0,1):\quad
\norm{u(\omega)}_{L^\infty(Q)}\leq 1-\delta(\omega)\,.
\]
This amounts to saying that almost every trajectory of $u$
its strictly separated from $\pm1$ in space and time. Nonetheless, 
note that the magnitude of separation depends on $\omega\in\Omega_0$,
hence identifies a positive random variable on $\Omega$.

\begin{remark}[$d=1$]
  Let us spend a few words on what happens in dimension $d=1$.
  In this case, in order to obtain the separation property on $u$, 
  the stronger regularity of Theorem~\ref{thm:reg} is not needed.
  Indeed, by simply assuming that $u_0\in L^2(\Omega; V)$ is separated from $\pm1$,
  one has that $u(\omega)\in C^0_w([0,T]; V)$ for 
  $\P$-almost every $\omega\in\Omega$, where 
  $V=H^1_0(\OO)\embed C^{0,\frac12}(\overline\OO)$.
  Hence, the pointwise evaluation as in \eqref{pointwise} holds
  and the proof of Theorem~\ref{thm2} can be easily adapted by
  substituting the space $Z$ with $V$: this only requires indeed that 
  $u_0\in L^2(\Omega; V)$ and
  the assumption that $s_0>2$.
\end{remark}

\subsection{Further developments}
It would be interesting to 
prove qualitative properties of the separation layer, such as
probability of separation of at least a given amount, 
exponential probability estimates, and convergence of 
the random separation layer towards the deterministic one 
as the noise vanishes. However, the main tool used in \cite{BOS},
namely the Bernstein inequality, cannot be directly applied in our case:
this is due to fact that the estimate on $\|G_{s_0}(u)\|_{L^1(\OO)}$
and $\|u\|_Z$ cannot be concatenated in a ``affine'' way, but 
they are in an exponential relation with one another. This seems 
a technical issue that requires further novel ideas.
Furthermore, from the modelling point of view, it would be 
relevant to investigate the separation property also for the stochastic 
Cahn-Hilliard equation. This requires an estimate in some H\"older-space
for $u$, at least in $L^\infty$ of time: however, the fourth-order structure 
of the equation prevents to use classical techniques 
that fit instead to second-order equations.

%%%%%%%%%%%%%%%%%%%%%%%%%%%%%%%%%%%%%%%%%

\section{Regularity}
\label{sec:reg}
In this section we prove Theorem~\ref{thm:reg}.
Let us recall that by \eqref{Gs} one has, for $s\geq1$,
\[
  G_s:(-1,1)\to\erre\,, \qquad G_s(x):=\frac1{(1-x^2)^s}\,, \quad x\in(-1,1)\,.
\]
Firstly, we give the following preliminary result,
which refines the ones proved in 
\cite[Thm.~2.2]{bertacco} and \cite[Lem.~3.1]{BOS}.

\begin{prop}\label{prop:est_Gs}
  Assume \ref{H1}--\ref{H3} and 
  \[
  u_0\in L^2(\Omega,\cF_0; V)\,, \qquad
  \exp\left(\|G_{s_0}(u_0)\|_{L^1(\OO)}\right) \in 
  L^q(\Omega,\cF_0)\quad\forall\,q\geq1\,.
  \]
  Then, the unique solution $u$ given by Proposition~\ref{prop:WP}
  also satisfies, for every $q\geq1$,
  \begin{equation}\label{est:Gs_exp}
  \E\sup_{t\in[0,T]}\int_\OO G_{s_0}(u(t))
  +\E\exp\left(q\norm{G_{s_0+1}(u)}_{L^1(0,T; L^1(\OO))}\right)
  <+\infty\,.
  \end{equation}
\end{prop}
\begin{proof}[Proof of Proposition~\ref{prop:est_Gs}]
  We would like to apply It\^o formula to $G_{s_0}(u)$.
  Due to the singularity at $\pm1$, this does not follow from the classical results
  but needs to be obtained via regularisation techniques.
  Precisely, by uniqueness of the solution $u$, $u$ can be seen
  as the limit in natural topologies 
  (the ones appearing in Proposition~\ref{prop:WP}) 
  of the sequence $(u_\lambda)_\lambda$, 
  where $u_\lambda\in L^2_\cP(\Omega; C^0([0,T]; H))$ is the unique solution to 
  \[
  u_\lambda(t) -\int_0^t\Delta u_\lambda(s)\,\d s + \int_0^tF_\lambda'(u_\lambda(s))\,\d s
  =u_0 + \int_0^t\mathcal H(J_\lambda(u_\lambda(s)))\,\d W(s)
  \]
  for every $t\in[0,T]$, $\P$-almost surely. Here,
  $F_\lambda':\erre\to\erre$ is a Yosida-type approximation of $F'$.
  For example, one can take $F_\lambda':=F'\circ J_\lambda$, where
  $J_\lambda:=\erre\to\erre$ is the resolvent of the nondecreasing 
  function $x\mapsto F'(x) + C_F x$, $x\in(-1,1)$.
  Analogously, 
  one could introduce the Moreau-Yosida-type approximation 
  ${G_{s_0,\lambda}}$ of $G_{s_0}$
  as $G_{s_0,\lambda}:=G_{s_0}\circ J_\lambda$.
  By It\^o formula we have then
  \begin{align*}
  &\int_\OO G_{s_0,\lambda}(u_\lambda(t)) 
  + \int_0^t\int_\OO G_{s_0,\lambda}''(u_\lambda(r))|\nabla u_\lambda(r)|^2\,\d r
  +\int_0^t\int_\OO F_\lambda'(u_\lambda(r))G_{s_0,\lambda}'(u_\lambda(r))\,\d r\\
  &\qquad=\int_\OO G_{s_0}(u_0)
  +\int_0^t\left(G_{s_0,\lambda}'(u_\lambda(r)), \H(J_\lambda(u_\lambda(r)))\,\d W(r)\right)_H\\
  &\qquad+\frac12\int_0^t\sum_{k=0}^\infty\int_\OO
  G_{s_0,\lambda}''(u_\lambda(r))|h_k(J_\lambda(u_\lambda(r)))|^2\,\d r
  \end{align*}
  for every $t\in[0,T]$, $\P$-almost surely. Now, the second term on the left-hand side 
  is nonnegative by monotonicity of $G_{s_0,\lambda}'$. 
  Moreover, by assumption \ref{H1},
  there are $\ell_1,\ell_2\in(-1,1)$ with 
  $\ell_1<\ell_2$ such that $F'(x)\leq -2$ for all $x\in(-1,\ell_1)$
  and $F'(x)\geq 2$ for all $x\in(\ell_2,1)$. Consequently, 
  noting also that $G_{s_0}'(x)\leq 0$ for all $x\in(-1,\ell_1)$
  and $G_{s_0}'(x)\geq 0$ for all $x\in(\ell_2,1)$,
  we have 
  \begin{align*}
   &\int_0^t\int_\OO F_\lambda'(u_\lambda(r))G_{s_0,\lambda}'(u_\lambda(r))\,\d r\\
   &=\int_0^t\int_\OO F'(J_\lambda(u_\lambda(r)))G_{s_0}'(J_\lambda(u_\lambda(r)))\,\d r\\
   &=\int_0^t\int_{\{-1<J_\lambda(u_\lambda(r))<\ell_1\}} F'(J_\lambda(u_\lambda(r)))
   G_{s_0}'(J_\lambda(u_\lambda(r)))\,\d r\\
   &\qquad+\int_0^t\int_{\{\ell_1\leq J_\lambda(u_\lambda(r))
   \leq \ell_2\}} F'(J_\lambda(u_\lambda(r)))
   G_{s_0}'(J_\lambda(u_\lambda(r)))\,\d r\\
   &\qquad+ \int_0^t\int_{\{\ell_2<J_\lambda(u_\lambda(r))<1\}}
   F'(J_\lambda(u_\lambda(r)))G_{s_0}'(J_\lambda(u_\lambda(r)))\,\d r\\
   &\geq 2\int_0^t\int_{\{J_\lambda(u_\lambda(r))\in(-1,\ell_1)\cup(\ell_2,1)\}}
   |G_{s_0}'(J_\lambda(u_\lambda(r)))|\,\d r
   -T|\OO|\max_{x\in[\ell_1,\ell_2]}|F'(x)||G_{s_0}'(x)|\\
   &\geq2\int_0^t\int_\OO|G_{s_0, \lambda}'(u_\lambda(r))|
   -T|\OO|\max_{x\in[\ell_1,\ell_2]}\left(|F'(x)||G_{s_0}'(x)| + 2|G_{s_0}'(x)|\right)\,.
  \end{align*}
  Eventually, noting that 
  $G_{s_0, \lambda}''(x)\leq \frac{C}{(1-J_\lambda(x)^2)^{s_0+2}}$ for all $x\in\erre$, and
  recalling assumption \ref{H3} and , we infer that 
  \begin{align*}
  &\frac12\int_0^t\sum_{k=0}^\infty\int_\OO
  G_{s_0,\lambda}''(u_\lambda(r))|h_k(J_\lambda(u_\lambda(r)))|^2\,\d r\\
  &\qquad\leq \frac C2\int_0^t\sum_{k=0}^\infty\int_\OO
  \frac{|h_k(J_\lambda(u_\lambda(r)))|^2}{|1+J_\lambda(u_\lambda(r))|^{s_0+2}
  |1-J_\lambda(u_\lambda(r))|^{s_0+2}}\,.
  \end{align*}
  On the right-hand side, 
  by Taylor's expansion with integral remainder and assumption \ref{H3}
  we get, for all $x\in[-1,1]$,
  \begin{align*}
    h_k(x)&=\int_{-1}^x
    \frac{h_k^{(s_0+2)}(\sigma)}{(s_0+1)!}(x-\sigma)^{s_0+1}\,\d \sigma\,,\\
    h_k(x)&=\int_{1}^x
    \frac{h_k^{(s_0+2)}(\sigma)}{(s_0+1)!}(x-\sigma)^{s_0+1}\,\d \sigma\,,
  \end{align*}
  so that, for all $x\in[-1,1]$,
  \begin{align*}
  |h_k(x)|&\leq \frac{\|h_k\|_{W^{1+s_0, \infty}(-1,1)}}{(s_0+2)!}|x+1|^{s_0+2}\,,\\
  |h_k(x)|&\leq \frac{\|h_k\|_{W^{1+s_0, \infty}(-1,1)}}{(s_0+2)!}|x-1|^{s_0+2}\,.
  \end{align*}
  Taking these estimates into account, we infer that
  \begin{align}
  \nonumber
  &\frac12\int_0^t\sum_{k=0}^\infty\int_\OO
  G_{s_0,\lambda}''(u_\lambda(r))|h_k(J_\lambda(u_\lambda(r)))|^2\,\d r\\
  \nonumber
  &\qquad\leq\frac C2\frac1{((s_0+2)!)^2}
  \int_0^t\sum_{k=0}^\infty\int_\OO
  \|h_k\|_{W^{1+s_0, \infty}(-1,1)}^2\\
  &\qquad\leq \frac C2\frac{C^2_{2,\H}}{((s_0+2)!)^2}
  T|\OO|\,.
  \label{est:aux}
  \end{align}
  We deduce that there exists a positive constant $C$, depending only on  
  $T$, $\OO$, $s_0$, $\ell_1$, $\ell_2$, $F'$, and $C_{2,\H}$, 
  and independent of $\lambda$,  such that 
  \begin{align}
  \label{est:aux3}
  \nonumber
  &\int_\OO G_{s_0,\lambda}(u_\lambda(t)) 
  +\int_0^t\int_\OO|G_{s_0,\lambda}'(u_\lambda(r))|\\
  &\qquad\leq C+\int_\OO G_{s_0}(u_0)
  +\int_0^t\left(G_{s_0,\lambda}'(u_\lambda(r)), \H(J_\lambda(u(r)))\,\d W(r)\right)_H\,.
  \end{align}
  Now, we note that 
  by analogous estimates as in \eqref{est:aux} 
  we get 
  \begin{align}
  \nonumber
  &\int_0^t\sum_{k=0}^\infty\int_\OO
  |G_{s_0,\lambda}'(u_\lambda(r))|^2|h_k(J_\lambda(u_\lambda(r)))|^2\,\d r\\
  \nonumber
  &\qquad\leq C\int_0^t\sum_{k=0}^\infty\int_\OO
  \frac{|h_k(J_\lambda(u_\lambda(r)))|^2}{|1+J_\lambda(u_\lambda(r))|^{2s_0+2}
  |1-J_\lambda(u_\lambda(r))|^{2s_0+2}}\\
  \nonumber
  &\qquad\leq C\frac1{((2s_0+2)!)^2}
  \int_0^t\sum_{k=0}^\infty\int_\OO
  \|h_k\|_{W^{1+2s_0, \infty}(-1,1)}^2\\
  \label{est:aux2}
  &\qquad\leq C\frac{C^2_{2,\H}}{((2s_0+2)!)^2}
  T|\OO|\,,
  \end{align}
  so that a direct application of the Burkholder-Davis-Gundy inequality 
  together with \eqref{est:aux2} yields 
  \[
  \E\sup_{t\in[0,T]}\int_\OO G_{s_0, \lambda}(u_\lambda(t)) 
  \leq C+\E\int_\OO G_{s_0}(u_0)\,,
  \] 
  and the first term of \eqref{est:Gs_exp} is estimated.
  Moreover, in order to handle the second term of \eqref{est:Gs_exp},
  we add and subtract  the term 
  \[
  \frac{\alpha}{2}\int_0^t\sum_{k=0}^\infty\int_\OO
  |G_{s_0, \lambda}'(u_\lambda(r))|^2|h_k(J_\lambda(u_\lambda(r)))|^2\,\d r\,,
  \]
  in \eqref{est:aux3}
  with $\alpha\geq1$ to be chosen later. 
  By virtue of \eqref{est:aux2} we infer that 
  there exists a constant $C_\alpha>0$ such that, setting 
  \begin{align*}
  M_{\alpha,\lambda}(t)&:=
  \int_0^t\left(G_{s_0,\lambda}'(u_\lambda(r)), \H(J_\lambda(u_\lambda(r)))\,\d W(r)\right)_H\\
  \quad&-\frac{\alpha}{2}\int_0^t\sum_{k=0}^\infty\int_\OO
  |G_{s_0, \lambda}'(u_\lambda(r))|^2|h_k(J_\lambda(u_\lambda(r)))|^2\,\d r\,, \quad t\in[0,T]\,,
  \end{align*}
  we have
  \begin{align*}
  &\int_\OO G_{s_0, \lambda}(u_\lambda(t)) 
  +\int_0^t\int_\OO|G_{s_0, \lambda}'(u_\lambda(r))|
  \leq C_\alpha+\int_\OO G_{s_0}(u_0)+ M_{\alpha,\lambda}(t)
  \quad\forall\,t\in[0,T]\,,\quad\P\text{-a.s.}\,.
  \end{align*}
  Hence, for every $q\geq1$ we have, by the Young inequality, 
  \begin{align*}
  &\exp\left(q\int_\OO G_{s_0, \lambda}(u_\lambda(t))\right)
  +\exp\left(q \int_0^t\int_\OO|G_{s_0,\lambda}'(u_\lambda(r))|\right)\\
  &\qquad\leq 2\exp \left[q\left(C_\alpha+\int_\OO G_{s_0}(u_0)\right)\right]
  \cdot\exp\left(qM_{\alpha,\lambda}(t)\right)\\
  &\qquad\leq \exp \left[2q\left(C_\alpha+\int_\OO G_{s_0}(u_0)\right)\right]
  +\exp\left(2qM_{\alpha,\lambda}(t)\right)
   \quad\forall\,t\in[0,T]\,,\quad\P\text{-a.s.}
  \end{align*}
  By choosing now $\alpha:=2q$, one has that 
  \[
  t\mapsto
  \exp\left(2qM_{2q,\lambda}(t)\right)\,, \quad t\in[0,T]\,,
  \]
  is a real positive supermartingale, so that for every $\lambda>0$
  \[
  \E\exp\left(2qM_{2q,\lambda}(t)\right)\leq 1 \quad\forall\, t\in[0,T]\,.
  \]
  Consequently, taking expectations and supremum in time yields
  \begin{align*}
  &\sup_{t\in[0,T]}\E\exp\left(q\int_\OO G_{s_0, \lambda}(u_\lambda(t))\right)
  +\E\exp\left(q \int_0^T\int_\OO|G_{s_0, \lambda}'(u_\lambda(r))|\right)\\
  &\qquad\leq 
  C_{q}\E\exp \left(2q\int_\OO G_{s_0}(u_0)\right)+1
  \end{align*}
  for some constant $C_q>0$. Now, by exploiting the 
  properties of the Moreau-Yosida approximation $G_{s_0, \lambda}$,
  by lower semicontinuity one gets 
  \[
  \exp\left(\norm{G_{s_0}'(u)}_{L^1(0,T; L^1(\OO))}\right)
  \in L^q(\Omega) \quad\forall\,q\geq1
  \]
  and the conclusion follows from the fact that 
  $|G_{s_0+1}|\leq C|G_{s_0}'|$ for a certain constant $C$ depending on $s_0$
  and the arbitrariness of $q\geq1$.
\end{proof}

\begin{proof}[Proof of Theorem~\ref{thm:reg}]
  Let us deal with the case $d=3$ at first.
  Here we will use the symbol $C$ to denote positive constants,
  only depending on the structural assumptions, whose values
  might be updated from line to line.
  For every $n\in\enne$, let 
  $P_n$ be the orthogonal projection on $H_n:=
  \operatorname{span}\{e_1,\ldots,e_n\}$, 
  where $(e_k)_k$ is a complete orthonormal system of $H$ made
  of eigenvectors of $-\Delta$ with homogeneous Dirichlet conditions.
  Then, by uniqueness of the solution $u$, $u$ can be seen
  as the limit in natural topologies 
  (the ones appearing in Proposition~\ref{prop:WP}) of the sequence $(u_n)_n$, 
  where $u_n\in L^2_\cP(\Omega; C^0([0,T]; H_n))$ is the unique solution to 
  \[
  u_n(t) -\int_0^t\Delta u_n(s)\,\d s + \int_0^tP_nF_n'(u_n(s))\,\d s
  =P_nu_0 + \int_0^tP_n\mathcal H(u_n(s))\,\d W(s)
  \]
  for every $t\in[0,T]$, $\P$-almost surely. Here, 
  $F_n':\erre\to\erre$ is a Yosida-type approximation of $F'$.
  Hence, the It\^o formula yields 
  \begin{align*}
  &\frac12\norm{\Delta u_n(t)}_H^2
  +\int_0^t\norm{\nabla\Delta u_n(s)}_H^2\,\d s
  +\int_0^t\int_\OO\Delta (P_nF_n'(u_n(s)))\Delta u_n(s)\,\d s\\
  &=\frac12\norm{\Delta (P_nu_0)}_H^2
  +\int_0^t\left(\Delta u_n(s), \Delta (P_n\mathcal H(u_n(s)))\,\d W(s)\right)_H\\
  &\qquad+\frac12\int_0^t\norm{\Delta (P_n\mathcal H(u_n(s)))}^2_{\cL^2(U,H)}\,\d s
  \end{align*}
  for every $t\in[0,T]$, $\P$-almost surely.
  Now, exploiting the fact that $u_n\in H_n$ almost everywhere, note that 
  \begin{align*} 
  \int_0^t\int_\OO\Delta (P_nF_n'(u_n(s)))\Delta u_n(s)\,\d s
  &=\int_0^t\int_\OO\Delta (F_n'(u_n(s)))\Delta u_n(s)\,\d s\\
  &=-\int_0^t\int_\OO\nabla F_n'(u_n(s))\cdot\nabla\Delta u_n(s)\,\d s\\
  &=-\int_0^t\int_\OO F_n''(u_n(s))\nabla u_n(s)\cdot\nabla\Delta u_n(s)\,\d s\,,
  \end{align*}
  so that by the H\"older and Young inequalities,
  as well as the inclusion $V\embed L^6(\OO)$,
  we get 
  \begin{align*}
  &\left|\int_0^t\int_\OO\Delta (P_nF_n'(u_n(s)))\Delta u_n(s)\,\d s\right|\\
  &\qquad\leq \int_0^t\norm{F_n''(u_n(s))}_{L^3(\OO)}
  \norm{\nabla u_n(s)}_{L^6(\OO)}
  \norm{\nabla\Delta u_n(s)}_H\,\d s\\
  &\qquad\leq \frac12\int_0^t
  \norm{\nabla\Delta u_n(s)}^2_H\,\d s
  +C\int_0^t\norm{F_n''(u_n(s))}^2_{L^3(\OO)}
  \norm{u_n(s)}_{Z}^2\,\d s
  \end{align*}
  for some positive constant $C$ depending only on $\OO$.
  Moreover, by \ref{H3} and the inclusion $Z\embed W^{1,4}(\OO)$
  we have that 
  \begin{align*}
  \int_0^t\norm{P_n\Delta \mathcal H(u_n(s))}^2_{\cL^2(U,H)}\,\d s
  &\leq\sum_{k=0}^\infty\int_0^t\norm{\Delta h_k(u_n(s))}_H^2\,\d s\\
  &=\sum_{k=0}^\infty\int_0^t
  \norm{h''_k(u_n(s))|\nabla u_n(s)|^2 + h_k'(u_n(s))\Delta u_n(s)}_H^2\,\d s\\
  &\leq 2C_{2,\H}^2\int_0^t\left(
  \norm{\nabla u_n(s)}^2_{L^4(\OO)} + \norm{\Delta u_n(s)}^2_H\right)\,\d s\\
  &\leq C\int_0^T\norm{u_n(s)}_Z^2\,\d s\,,
  \end{align*}
  possibly updating the value of the constant $C$. 
  Hence, taking supremum in time in It\^o formula and
  rearranging the terms, we have 
  \begin{align*}
    \sup_{r\in[0,t]}\norm{u_n(r)}_Z^2
    +\int_0^t\norm{u_n(s)}_{H^3(\OO)}^2\,\d s
    &\leq C\norm{u_0}_Z^2 + 
    C\sup_{r\in[0,T]}\left|
    \int_0^r\left(\Delta u_n(s), \Delta \mathcal H(u_n(s))\,\d W(s)\right)_H
    \right|\\
    &\quad+ C
    \int_0^t\left(1+\norm{F_n''(u(s))}^2_{L^3(\OO)}\right)
    \sup_{r\in[0,s]}\norm{u_n(r)}_{Z}^2\,\d s\,,
  \end{align*}
  and the Gronwall lemma yields 
  \begin{align*}
    &\sup_{t\in[0,T]}\norm{u_n(t)}_Z^2
    +\int_0^T\norm{u_n(s)}_{H^3(\OO)}^2\,\d s\\
    &\qquad\leq C\left[
    \norm{u_0}_Z^2 + 
    \sup_{r\in[0,T]}\left|
    \int_0^r\left(\Delta u_n(s), \Delta \mathcal H(u_n(s))\,\d W(s)\right)_H
    \right|\right]
    \exp\left(\norm{F_n''(u_n)}^2_{L^2(0,T; L^3(\OO))}\right)\,.
  \end{align*}
Let now $\ell\in(2,\min\{4,p,\frac{4p}{2+p}\})$:
the Young inequality implies that 
\begin{align*}
    &\E\sup_{t\in[0,T]}\norm{u_n(t)}_Z^2
    +\E\int_0^T\norm{u_n(s)}_{H^3(\OO)}^2\,\d s\\
    &\qquad\leq C\left[
    \norm{u_0}_{L^\ell(\Omega; Z)}^\ell+
    \E\sup_{r\in[0,T]}\left|
    \int_0^r\left(\Delta u_n(s), \Delta \mathcal H(u_n(s))\,\d W(s)\right)_H
    \right|^{\frac \ell2}\right]\\
    &\qquad+ C\E\exp\left(\frac \ell{\ell-2}\norm{F_n''(u_n)}^2_{L^2(0,T; L^3(\OO))}\right)\,.
\end{align*}
Now,
  by the Burkholder-Davis-Gundy and Young inequalities,
  analogous computations show that 
  \begin{align*}
  &\E\sup_{t\in[0,T]}\left|
  \int_0^t\left(\Delta u_n(s), \Delta \mathcal H(u_n(s))\,\d W(s)\right)_H
  \right|^{\frac \ell2}\\
  &\qquad\leq C\E\left[\left(\int_0^T\norm{\Delta u_n(s)}_H^2
  \norm{\Delta \H(u_n(s))}_{\cL^2(U,H)}^2\,\d s\right)^{\frac \ell4}\right]\\
  &\qquad\leq
  C\E\left[\sup_{t\in[0,T]}\norm{\Delta u_n(t)}_H^{\frac \ell2}
  \norm{\Delta\H(u_n)}_{L^2(0,T;\cL^2(U,H))}^{\frac \ell2}\right]\\
  &\qquad\leq\frac14\E\sup_{t\in[0,T]}\norm{\Delta u_n(t)}_H^2
  +C\E\norm{u_n}^{\frac{2\ell}{4-\ell}}_{L^2(0,T;Z)}\,,
  \end{align*}
  possibly updating the value of the constant $C$. Hence, 
  noting that $\frac{2\ell}{4-\ell}<p$, we infer that 
  \begin{align*}
    &\E\sup_{t\in[0,T]}\norm{u_n(t)}_Z^2
    +\E\int_0^T\norm{u_n(s)}_{H^3(\OO)}^2\,\d s\\
    &\qquad\leq C\left[1+
    \norm{u_0}_{L^p(\Omega; Z)}^p
    +\norm{u_n}_{L^p(\Omega; L^2(0,T; Z))}^p\right]
    + C\E\exp\left(\frac \ell{\ell-2}\norm{F_n''(u_n)}^2_{L^2(0,T; L^3(\OO))}\right)\,.
\end{align*}
Now, by Propostion~\ref{prop:est_Gs}
it holds that $\exp (\|G_{s_0+1}(u)\|_{L^1(0,T; L^1(\OO))})\in L^q(\Omega)$
for every $q\geq1$. Since $s_0\geq3s_F-1$, this implies that 
$\exp (\|G_{3s_F}(u)\|_{L^1(0,T; L^1(\OO))})\in L^q(\Omega)$
for every $q\geq1$. Consequently, by assumption \ref{H1} we deduce that 
$\exp (\|F''(u)\|_{L^3(0,T; L^3(\OO))})\in L^q(\Omega)$
for every $q\geq1$: thanks to the proof of
Propostion~\ref{prop:est_Gs}, this implies also
that $\|\exp (\|F_n''(u_n)\|_{L^3(0,T; L^3(\OO))})\|_{L^q(\Omega)}\leq C_q$
for every $q\geq1$, where the constant $C_q$ does not depend on $n$.
Moreover, since $u\in L^p(\Omega; L^2(0,T; Z))$ one gets analogously that 
$\norm{u_n}_{L^p(\Omega; L^2(0,T; Z))}\leq C$.
Hence, the last term on the right-hand side 
is bounded in $n$, and this concludes the proof in 
dimension $d=3$.
As far as dimension $d=2$ is concerned, the proof can
be adapted by taking into account that $Z\embed W^{1,r}(\OO)$
for every $r\geq1$, so that 
  \begin{align*}
  &\left|\int_0^t\int_\OO\Delta (P_nF_n'(u_n(s)))\Delta u_n(s)\,\d s\right|\\
  &\qquad\leq \int_0^t\norm{F_n''(u_n(s))}_{L^{\frac{2r}{r-2}}(\OO)}
  \norm{\nabla u_n(s)}_{L^r(\OO)}
  \norm{\nabla\Delta u_n(s)}_H\,\d s\\
  &\qquad\leq \frac12\int_0^t
  \norm{\nabla\Delta u_n(s)}^2_H\,\d s
  +C\int_0^t\norm{F_n''(u_n(s))}^2_{L^{\frac{2r}{r-2}}(\OO)}
  \norm{u_n(s)}_{Z}^2\,\d s\,.
  \end{align*}
  All the remaining terms are handled analogously to the case $d=3$.
  In this case, noting that one can choose $r$ sufficiently large so that
  $\frac{2r}{r-2}\leq2$, the condition $s_0\geq 2s_F-1$ in  
  assumption \ref{H1} guarantees that 
  $\exp (\|F''(u)\|_{L^2(0,T; L^2(\OO))})\in L^q(\Omega)$
  for every $q\geq1$. This allows to conclude in the same way.
\end{proof}

%%%%%%%%%%%%%%%%%%%%%%%%%%%%%%%%%%%%%%%%%%%%%%

\section{Separation property}
\label{sec:proof2}
In this section we prove Theorem~\ref{thm2} on the random separation property.
The argument is similar to the one in \cite{BOS}.
More precisely, in a first place we show that 
$|u(t, x)| < 1$ for every $(t, x) \in (0, T) \times \OO$
on almost every trajectory,
namely that $u$ cannot touch the barriers $\pm1$
 at any point in space and time.
Secondly, we prove by contradiction the existence of a
random separation layer $\delta = \delta(\omega) > 0$
such that $\|u\|_{L^{\infty}(Q)} \leq 1 - \delta$ 
on almost every trajectory. 

\begin{lem}\label{lem:sep}
  Under the assumptions of Theorem~\ref{thm2}, it holds that 
  \[
  |u(\omega,t,x)|<1 \quad\forall\,(t,x)\in[0,T]\times\overline\OO\,, 
  \quad\text{for $\P$-a.e.~}\omega\in\Omega\,.
  \]
\end{lem}
\begin{proof}
  By the regularity Theorem~\ref{thm:reg}, one has that 
  \[
  u\in L^2_\cP(\Omega; C^0([0,T]; V))\cap L^2_w(\Omega; L^\infty(0,T; Z))\,,
  \]
  while Proposition~\ref{prop:est_Gs} ensures that 
  \[
  E\sup_{t\in[0,T]}\int_\OO G_{s_0}(u(t))<+\infty\,.
  \]
  This implies that there exists $\Omega_0\in\cF$ with $\P(\Omega_0)=1$ 
  such that
  \[
  u(\omega)\in C^0([0,T]; V)\cap L^\infty(0,T; Z)
  \subset C^0_w([0,T]; Z) 
  \quad\forall\,\omega\in\Omega_0
  \]
  and
  \begin{equation}\label{Gs_infty}
  \sup_{t\in[0,T]}
  \norm{G_{s_0}(u(\omega,t))}_{L^1(\OO)}<+\infty
  \quad\forall\,\omega\in\Omega_0\,.
  \end{equation}
  Now, by the Sobolev embedding theorem  
  we have the continuous inclusion 
  \[
  Z\embed C^{0,\alpha}(\overline\OO)\,,
  \]
  where $\alpha\in(0,\frac12)$ if $d=3$ and $\alpha\in(0,1)$ if $d=2$. Hence, one has
  \[
  u(\omega)\in C^0_w([0,T]; C^{0,\alpha}(\overline\OO)) 
  \quad\forall\,\omega\in\Omega_0\,,
  \]
  so that the pointwise evaluation $u(\omega,t,x)$ is well-defined 
  for every $(\omega,t,x)\in\Omega_0\times[0,T]\times\overline\OO$,
  and it holds that 
  \begin{equation}\label{c_weak}
  \sup_{t\in[0,T]}
  \norm{u(\omega,t)}_{C^{0,\alpha}(\overline\OO)}<+\infty
  \quad\forall\,\omega\in\Omega_0\,.
  \end{equation}
  Let now fix $\omega\in\Omega_0$ and suppose
  by contradiction that there exists 
  $(t_0,x_0)\in[0,T]\times\overline\OO$
 such that $|u(\omega, t_0, x_0)|=1$. By noting that $x_0\in\OO$
 due to the boundary conditions, we have that
\begin{align*}
  \sup_{t\in[0,T]}\norm{G_{s_0}(u(\omega,t))}_{L^1(\OO)}
  &\geq \int_\OO G_{s_0}(u(\omega, t_0, x))\,\d x=
  \int_{\OO} \frac{1}{\left|1-u(\omega,t_0, x)^2\right|^{s_0}} \d x \\
  &=\int_{\OO} \frac{1}{\left|u(\omega,t_0, x_0)^2
  -u(\omega,t_0, x)^2\right|^{s_0}} \d x\\
  &=\int_{\OO} \frac{1}{\left|u(\omega, t_0, x_0)
  -u(\omega,t_0, x)\right|^{s_0}}
  \frac1{\left|u(\omega, t_0, x_0)+u(\omega,t_0, x)\right|^{s_0}} \d x\,.
\end{align*}
Now, on the one hand we have
\[
\left|u(\omega, t_0, x_0)+
u(\omega, t_0, x)\right|^{s_0}
 \leq \left(1 + |u(\omega,t_0, x)|\right)^{s_0} 
 \leq 2^{s_0} \quad \forall\,x \in \OO\,,
\]
and on the other hand we have
\begin{align*}
|u(\omega, t_0, x_0) - u(\omega, t_0, x)|^{s_0}
&\leq \norm{u(\omega, t_0)}^{s_0}_{C^{0, \alpha}(\overline\OO)} 
|x-x_0|^{\alpha s_0}\\
&\leq \sup_{t\in[0,T]}\norm{u(\omega,t)}_{C^{0, \alpha}(\overline\OO)}^{s_0}
|x-x_0|^{\alpha s_0} \,.
\end{align*}
Taking into account \eqref{Gs_infty} and \eqref{c_weak} yields 
\begin{align*}
+\infty>\sup_{t\in[0,T]}\norm{G_{s_0}(u(\omega,t))}_{L^1(\OO)}\cdot
\sup_{t\in[0,T]}\norm{u(\omega,t)}_{C^{0, \alpha}(\overline\OO)}^{s_0}
& \geq 2^{-{s_0}}
\int_{\OO} \frac{1}{|x-x_0|^{\alpha s_0}}\, \d x\,.
\end{align*}
However, recalling that by the assumptions of Theorem~\ref{thm2}
one has $s_0>2$ if $d=2$ and $s_0>6$ if $d=3$, one readily checks
that there exists $\alpha$ such that
$\alpha s_0> d$. This implies that the right-hand side diverges,
hence yields a contradiction.  
\end{proof}
 
\begin{proof}[Proof of Theorem~\ref{thm2}]
Let $\omega\in\Omega_0$ be fixed, where
$\Omega_0$ is given by the proof of Lemma~\ref{lem:sep}.
By contradiction, 
suppose that there exists a
sequence $(t_n, x_n)_{n \in \mathbb{N}} \subset [0, T] \times \overline\OO$
such that $|u(\omega, t_n , x_n)| \to 1$ as $n\to\infty$. 
It is not restrictive to suppose that the entire sequence satisfies 
$(t_n, x_n)\to (t_0, x_0)$ as $n\to\infty$
for some $(t_0, x_0)\in[0,T]\times\overline\OO$.
Otherwise, this follows by standard compactness arguments on 
a subsequence. Given $\alpha$ as in the proof of Lemma~\ref{lem:sep}, we have
\begin{align*}
|u(\omega, t_{n}, x_{n}) - u(\omega, t_0, x_0)| 
&\leq |u(\omega, t_{n}, x_{n}) - u(\omega, t_{n}, x_0)| 
+ |u(\omega, t_{n}, x_0) - u(\omega, t_0, x_0)| \\
& \leq
\sup_{t\in[0,T]}\|u(\omega, t)\|_{C^{0,\alpha}(\overline\OO)} |x_{n} - x_0|^{\alpha} 
+ |u(\omega, t_{n}, x_0) - u(\omega, t_0, x_0)|\,.
\end{align*}
The first term on the right-hand side converges to $0$ as $n\to\infty$
thanks to \eqref{c_weak}.
Moreover, for the second term, by introducing the
continuous linear functional
\[
  I_{x_0}: C^{0,\alpha}(\overline\OO)\to\erre\,, \qquad
  \langle I_{x_0}, v\rangle:= v(x_0)\,,\quad v\in C^{0,\alpha}(\overline\OO)\,,
\]
 one has
\[
  |u(\omega, t_{n}, x_0) - u(\omega, t_0, x_0)|
  =|\langle I_{x_0}, u(\omega,t_n)\rangle - \langle I_{x_0}, u(\omega,t_0)\rangle|\,.
\]
Since $u(\omega)\in C^0_w([0,T]; C^{0,\alpha}(\overline\OO))$, 
it holds as $n\to\infty$ that 
\[
|u(\omega, t_{n}, x_0) - u(\omega, t_0, x_0)| \to 0\,.
\]
In conclusion, these remarks imply that 
$|u(\omega, t_0, x_0)|=1$, which is in contradiction with Lemma~\ref{lem:sep}.
Hence, we have shows that there exists 
$\delta(\omega)>0$
such that $|u(\omega, t,x)|\leq 1 - \delta(\omega)$ for every 
$(t,x)\in[0,T]\times\overline\OO$. 
Clearly, $\delta(\omega)\in(0,\delta_0]$ for every $\omega\in \Omega_0$
due the assumption on the initial condition in 
Theorem~\ref{thm2}.
\end{proof}

%%%%%%%%%%%%%%%%%%%%%%%%%%%%%%%%%%%%%%%%%%

\section*{Acknowledgement}
The authors are members of Gruppo Nazionale 
per l'Analisi Matematica, la Probabilit\`a 
e le loro Applicazioni (GNAMPA), 
Istituto Nazionale di Alta Matematica (INdAM), 
and gratefully acknowledge financial support 
through the project CUP\_E55F22000270001.
The present research has also been supported by 
the MUR grant  ``Dipartimento di
eccellenza 2023--2027'' for Politecnico di Milano.

%%%%%%%%%%%%%%%%%%%%%%%%%%%%%%%%%%%%%%%%%%

\bibliography{ref}{}
\bibliographystyle{abbrv}

\end{document}